\documentclass[11pt,twoside]{amsart}

\usepackage{amsmath}
\usepackage{amsthm}
\usepackage{amsfonts, amssymb}
\usepackage{mathrsfs}
\usepackage[all]{xy}
\usepackage{url}

\usepackage{latexsym}
\usepackage[dvips]{graphicx}

\theoremstyle{plain}
\newtheorem{lema}{Lemma}[section]
\newtheorem{prop}[lema]{Proposition}
\newtheorem{teo}[lema]{Theorem}

\newtheorem*{intro1}{Theorem \ref{triangle}}
\newtheorem*{intro2}{Theorem \ref{conn}}
\newtheorem*{intro3}{Corollary \ref{cat}}
\newtheorem{coro}[lema]{Corollary}
\theoremstyle{remark}

\newtheorem{obs}[lema]{Remark}

\theoremstyle{definition}
\newtheorem{defi}[lema]{Definition}
\newtheorem{ej}[lema]{Example}

\begin{document}

\title[Star clusters in independence complexes of graphs]{Star clusters in independence complexes of graphs}

\author[J.A. Barmak]{Jonathan Ariel Barmak $^{\dagger}$}

\thanks{$^{\dagger}$ Supported by grant KAW 2005.0098 from the Knut 
and Alice Wallenberg Foundation.}

\address{Mathematics Department\\
Kungliga Tekniska h\"ogskolan\\
Stockholm, Sweden}

\email{jbarmak@kth.se}

\begin{abstract}
We introduce the notion of \textit{star cluster} of a simplex in a simplicial complex. This concept provides a general tool to study the topology of independence complexes of graphs. We use star clusters to answer a question arisen from works of Engstr\"om and Jonsson on the homotopy type of independence complexes of triangle-free graphs and to investigate a large number of examples which appear in the literature. We present an alternative way to study the chromatic number of a graph from a homotopical point of view and obtain new results regarding the connectivity of independence complexes. 
\end{abstract}



\maketitle

\section{Introduction} \label{sectionintro}

Since Lov\'asz' proof of the Kneser conjecture in 1978, numerous applications of algebraic topology to combinatorics, and in particular to graph theory, have been found. A recurrent strategy in topological combinatorics consists in the study of homotopy invariants of certain CW-complexes constructed from a discrete structure to obtain combinatorial information about the original object. In Lov\'asz' prototypical example, connectivity properties of the \textit{neighborhood complex} $\mathcal{N}(G)$ of a graph $G$ are shown to be closely related to the chromatic number $\chi (G)$ of $G$. Lov\'asz conjecture states that there exists a similar relationship between the so called \textit{Hom complexes} Hom$(H, G)$ and $\chi (G)$ when $H$ is a cycle with an odd number of vertices. The Hom complex Hom$(H,G)$ is homotopy equivalent to $\mathcal{N}(G)$ when $H$ is the complete graph on two vertices $K_2$. Babson and Kozlov \cite{Bab} proved this conjecture in 2007. In their proof they used that Hom$(G,K_n)$ is linked to another polyhedron associated to $G$, which is called the \textit{independence complex} of $G$. Given a graph $G$, its independence complex $I_G$ is the simplicial complex whose simplices are the independent sets of vertices of $G$. In this approach to the conjecture it was then needed to understand the topology of independence complexes of cycles.

For any finite simplicial complex $K$ there exists a graph $G$ such that $I_G$ is homeomorphic to $K$. Specifically, given a complex $K$, we consider the graph $G$ whose vertices are the simplices of $K$ and whose edges are the pairs $(\sigma, \tau)$ of simplices such that $\sigma$ is not a face of $\tau$ and $\tau$ is not a face of $\sigma$. Then $I_G$ is isomorphic to the barycentric subdivision $K'$ of $K$. In particular, the homotopy types of independence complexes of graphs coincide with homotopy types of compact polyhedra.

In the last years a lot of attention has been drawn to study the general problem of determining all the possible homotopy types of the independence complexes of graphs in some particular class. For instance, Kozlov \cite{Koz} investigates the homotopy types of independence complexes of cycles and paths, Ehrenborg and Hetyei \cite{Ehr} consider this question for forests, Engstr\"om \cite{Eng2} for claw-free graphs, Bousquet-M\'elou, Linusson and Nevo \cite{Bou2} for some square grids, Braun \cite{Bra} for Stable Kneser graphs and Jonsson \cite{Jon} for bipartite graphs. Other results investigate how the topology of the independence complex changes when the graph is modified in some particular way. Engstr\"om \cite{Eng} analyzes what happens when some special points of the graph are removed and Csorba \cite{Cso} studies how subdivisions of the edges of a graph affect the associated complex.
 
The purpose of this paper is two-fold: to introduce a notion that allows the development of several techniques which are useful to study homotopy types of independence complexes, and to establish new relationships between combinatorial properties of graphs and homotopy invariants of their independence complexes. We have mentioned that independence complexes are closely related to Hom complexes and therefore, they can be used to study chromatic properties of graphs. In this paper we show in a direct way how to use independence complexes to investigate colorability of graphs. We will obtain lower bounds for the chromatic number of a graph in terms of a numerical homotopy invariant associated to its independence complex. On the other hand we will introduce some ideas that are used to study the connectivity of $I_G$ in terms of combinatorial properties of $G$.  

\medskip 

One of the motivating questions of this work appears in Engstr\"om's Thesis \cite{Eng3} and concerns the existence of torsion in the homology groups of independence complexes of triangle-free graphs (i.e. graphs which do not contain triangles). Recently, Jonsson \cite{Jon} proved that for any finitely generated abelian group $\Gamma$ and any integer $n\ge 2$, there exists a triangle-free graph $G$ such that the (integral) homology group $H_n(I_G)$ is isomorphic to $\Gamma$. In fact, he shows that the homotopy types of independence complexes of bipartite graphs are exactly the same as the homotopy types of suspensions of compact polyhedra. Two natural questions arise from this work. Can $H_1(I_G)$ have torsion for some triangle-free graph $G$? And furthermore, what are the homotopy types of independence complexes of triangle-free graphs? In order to give a solution to these problems we introduce the notion of \textit{star cluster} of a simplex in a simplicial complex. The star cluster $SC(\sigma)$ of a simplex $\sigma \in K$ is just the union of the simplicial stars of the vertices of $\sigma$. In general these subcomplexes can have non-trivial homotopy type but we will see that if $K$ is the independence complex of a graph, then the star cluster of every simplex is contractible (Lemma \ref{none}). These fundamental blocks are used to answer both questions stated above. We prove that the homotopy types of complexes associated to triangle-free graphs also coincide with those of suspensions. In fact we show the following stronger result. 

\begin{intro1}
Let $G$ be a graph such that there exists a vertex $v\in G$ which is contained in no triangle. Then the independence complex of $G$ has the homotopy type of a suspension. In particular, the independence complex of any triangle-free graph has the homotopy type of a suspension.
\end{intro1}

From this it is immediately deduced that $H_1(I_G)$ is a free abelian group for every triangle-free graph $G$.

We will see that fortunately, these results are just the first application of star clusters. The fact of being contractible, makes these subcomplexes suitable for developing general tools to attack problems regarding independence complexes. We use star clusters to give alternative and shorter proofs of various known results. Many of the original proofs use Forman's discrete Morse theory or the Nerve lemma \cite[Theorem 10.6]{Bjo}. Star clusters provide a much more basic technique to deal with these and other problems.

The \textit{matching complex} $M_n$ of $K_n$ and the \textit{chessboard complex} $M_{n,m}$ appear in many different contexts in mathematics. The chessboard complex was first considered by Garst \cite{Gar} in connection with Tits coset complexes, and $M_n$ was studied by Bouc in \cite{Bou3} where he worked with Quillen complexes. The homotopy type of these complexes is not completely determined although sharp bounds for the connectivity are known \cite{Bjo2, Sha}. $M_n$ and $M_{n,m}$ are some examples of the larger class of matching complexes. This class is in turn contained in the class of independence complexes of \textit{claw-free graphs}. In \cite{Eng2} Engstr\"om gives a bound for the connectivity of independence complexes of claw-free graphs in terms of the number of vertices and the maximum degree. In Section \ref{sectionmatchings} of this article we use star clusters to prove a sharp bound for the connectivity of matching complexes and independence complexes of claw-free graphs which depends only on the dimension of the complexes.  

\begin{intro2}
Let $G$ be a claw-free graph. Then $I_G$ is $[\frac{dim(I_G)-2}{2}]$-connected.
\end{intro2}

From this result one can deduce that the homology of those complexes is non-trivial only for degrees contained in an interval of the form $\{[\frac{n}{2}], \ldots , n\}$. These techniques are also used to give results of connectivity for general graphs.

The \textit{neighborhood complex} $\mathcal{N}(G)$ mentioned at the beginning is a simplicial complex that one can associate to a given graph $G$. The simplices of $\mathcal{N}(G)$ are the sets of vertices that have a common neighbor. One of the key points of Lov\'asz' celebrated proof of the Kneser conjecture \cite{Lov} is the following result which relates the chromatic number $\chi(G)$ with the connectivity $\nu (\mathcal{N}(G))=max \{ n \ | \ \mathcal{N}(G)$ is $n$-connected$\}$ of the neighborhood complex.

\begin{teo} [Lov\'asz] \label{lov}
Let $G$ be a graph. Then $\chi (G) \ge \nu (\mathcal{N}(G))+3$.
\end{teo} 

In Section \ref{sectionls} we study the relationship between the chromatic number of a graph and the topology of its independence complex. The \textit{Strong Lusternik-Schnirelmann category} $Cat(X)$ of a space $X$ is one less that the minimum number of contractible subcomplexes which are needed to cover some CW-complex homotopy equivalent to $X$. This homotopy invariant is not easy to determine in concrete examples. We prove the following result as another application of star clusters:

\begin{intro3}
Let $G$ be a graph. Then $\chi (G) \ge Cat(I_G)+1$. 
\end{intro3}

In particular we obtain non-trivial bounds for chromatic numbers of graphs whose independence complexes are homotopy equivalent to a projective space or a torus.
   
In the last section of the paper we introduce a construction which generalizes Csorba's edge subdivision. This is used to give an alternative proof of a recent result by Skwarski \cite{Skw} regarding planar graphs and to obtain new results about homotopy types of independence complexes of graphs with bounded maximum degree. 

\bigskip

\section{Preliminaries}   

In this section we recall some basic results and introduce the notation that will be needed in the rest of the article.
All the graphs considered will be simple (undirected, loopless and without parallel edges) and finite. All the simplicial complexes we work with are finite. Many times we will use just the word ``complex" to refer to these objects. We will confuse a simplicial complex with its geometric realization. If two complexes $K$ and $L$ are homotopy equivalent, we will write $K\simeq L$. The (non-reduced) suspension of a topological space $X$ is denoted as usual by $\Sigma(X)$. The \textit{(simplicial) star} $st_K(\sigma)$ of a simplex $\sigma$ in a complex $K$ is the subcomplex of simplices $\tau$ such that $\tau \cup \sigma \in K$. The star of a simplex is always a cone and therefore contractible. When there is no risk of confusion we will omit the subscripts in the notation.

\begin{defi}
We say that a complex $K$ is \textit{clique} if for each non-empty set of vertices $\sigma$ such that $\{v,w\} \in K$ for every $v,w\in \sigma$, we have that $\sigma \in K$.   
\end{defi}

The \textit{clique complex} of a graph $G$ is a simplicial complex whose simplices are the cliques of $G$, that is, the subsets of pairwise adjacent vertices of $G$. Then, a complex $K$ is clique if and only if it is the clique complex of some graph. In fact, if $K$ is clique, it is the clique complex of its $1$-skeleton $K^1$.

The \textit{independence complex} $I_G$ of a graph $G$ is the simplicial complex whose simplices are the independent subsets of vertices of $G$. In other words, it is the clique complex of the complementary graph $\overline{G}$. Therefore a complex $K$ is the independence complex of some graph if and only if it is clique.

If $\sigma$ is an independent set in a graph $G$ and $v$ is a vertex of $G$ such that $\sigma \cup \{v\}$ is also independent, we will say that $\sigma$ \textit{can be extended} to $v$. This is equivalent to say that $\sigma \in st_{I_G}(v)$ when $\sigma$ is non-empty.

\begin{obs}
If a graph $G$ is the disjoint union of two graphs $H_1$ and $H_2$, then its independence complex $I_G$ is the (simplicial) join $I_{H_1}*I_{H_2}$. In particular if $H_1$ is just a point, $I_G$ is the simplicial cone with base $I_{H_2}$ and if $H_1$ is an edge, $I_G=\Sigma (I_{H_2})$.
\end{obs}

Recall that if $X_1, X_2$ and $Y$ are three topological spaces and the first two have the same homotopy type, then $X_1*Y\simeq X_2*Y$.

A basic result in topology, sometimes called gluing theorem, says that if

\begin{displaymath}
\xymatrix@C=20pt{ A \ar@{->}^f[r] \ar@{->}^i[d] & Y \ar@{->}[d]\\
		  X \ar@{->}^{\overline{f}}[r] & Z } 
\end{displaymath}
\\[0.5cm]
is a push-out of topological spaces, $f$ is a homotopy equivalence and $i$ is a closed cofibration ($A$ is a closed subspace of $X$ and $(X,A)$ has the homotopy extension property), then $\overline{f}$ is also a homotopy equivalence. For a proof of this result the reader can see \cite[7.5.7 (Corollary 2)]{Bro}. If $K$ is a simplicial complex and $L\subseteq K$ is a subcomplex, the inclusion $L\hookrightarrow K$ is a closed cofibration. This is in fact the unique type of cofibrations that we will work with here. We can use this result to prove that the quotient $K/L$ of a complex $K$ by a contractible subcomplex $L$ is homotopy equivalent to $K$ or that the union of two contractible complexes is contractible if the intersection is contractible. These applications will appear in some proofs below (Lemma \ref{none} and Theorem \ref{kneser}). We will also use the gluing theorem in the proof of Lemma \ref{dos}.  
  
\bigskip

\section{Star clusters and triangle-free graphs}

\begin{defi}
Let $\sigma$ be a simplex of a simplicial complex $K$. We define the \textit{star cluster} of $\sigma$ in $K$ as the subcomplex $$SC_K(\sigma)=\bigcup\limits_{v \in \sigma} st_K(v).$$
\end{defi}

\begin{lema} \label{none}
Let $K$ be a clique complex. Let $\sigma$ be a simplex of $K$ and let $\sigma_0, \sigma_1, \ldots ,\sigma_r$ be a collection of faces of $\sigma$ ($r\ge 0$). Then $$\bigcup\limits_{i=0}^r \bigcap\limits_{v \in \sigma_i} st_K(v)$$ is a contractible subcomplex of $K$. In particular, the star cluster of a simplex in a clique complex is contractible.
\end{lema}
\begin{proof}
By the clique property $$\bigcap\limits_{v \in \sigma_0} st_K(v)=st_K(\sigma_0),$$ which is contractible. Then the statement is true for $r=0$. Now assume that $r$ is positive. In order to prove that the union of the complexes $$K_1=\bigcup\limits_{i=0}^{r-1} \bigcap\limits_{v \in \sigma_i} st_K(v), \ K_2=\bigcap\limits_{v \in \sigma_{r}} st_K(v)$$ is contractible, it suffices to show that each of them and the intersection are. But $$K_1\cap K_2=\bigcup\limits_{i=0}^{r-1} \bigcap\limits_{v \in \sigma_i \cup \sigma_{r}} st_K(v),$$ so by induction all three complexes $K_1, K_2$ and $K_1\cap K_2$ are contractible, and then so is $K_1 \cup K_2$.

To deduce that the star cluster of a simplex $\sigma$ in a clique complex is contractible, it suffices to take the collection $\{ \sigma_i \}$ as the set of $0$-dimensional faces of $\sigma$.
\end{proof}


In fact it can be proved that star clusters in clique complexes are collapsible. Moreover, they are non-evasive.

\begin{figure}[h] 
\begin{center}
\includegraphics[scale=0.4]{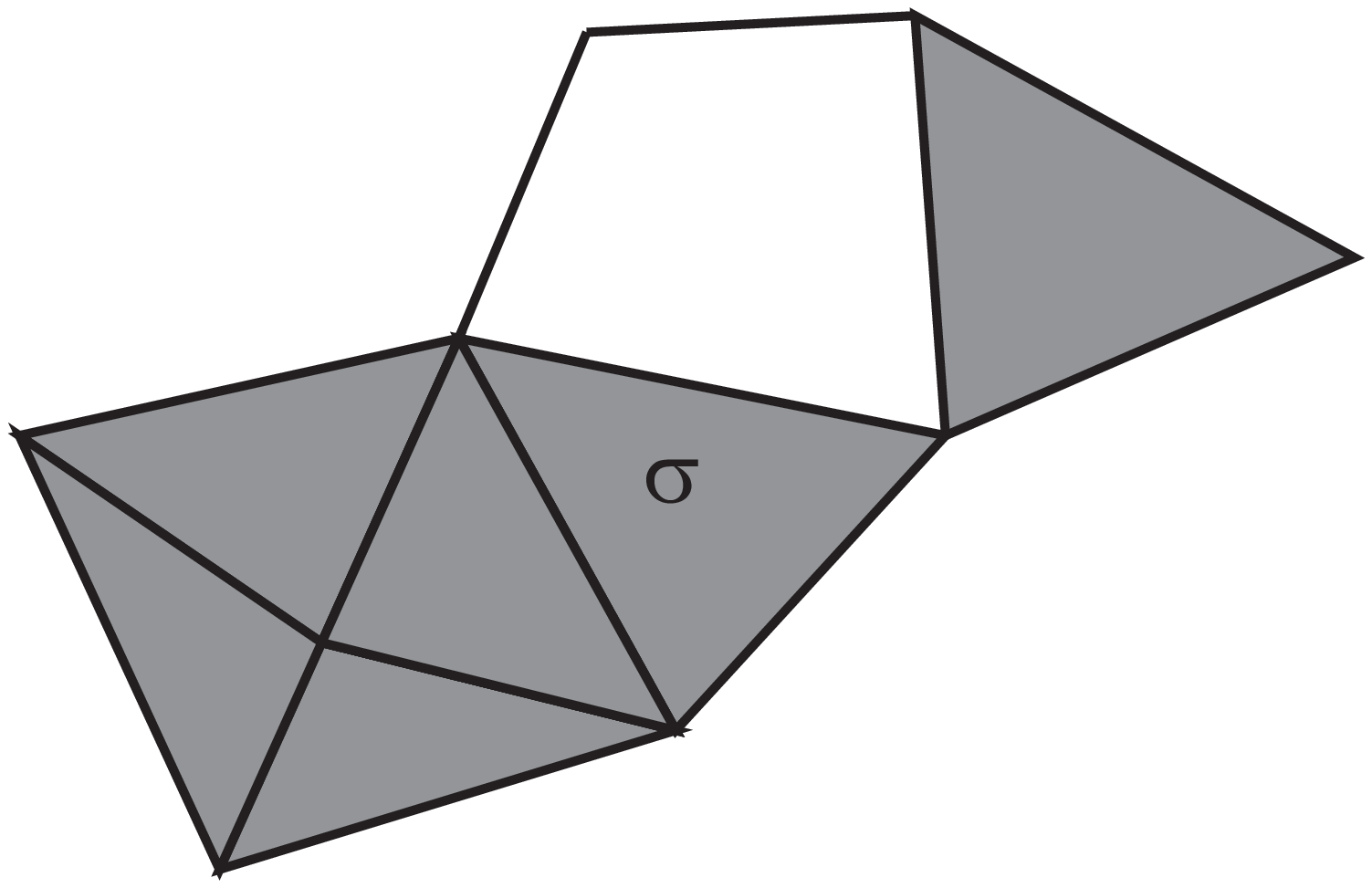} \includegraphics[scale=0.4]{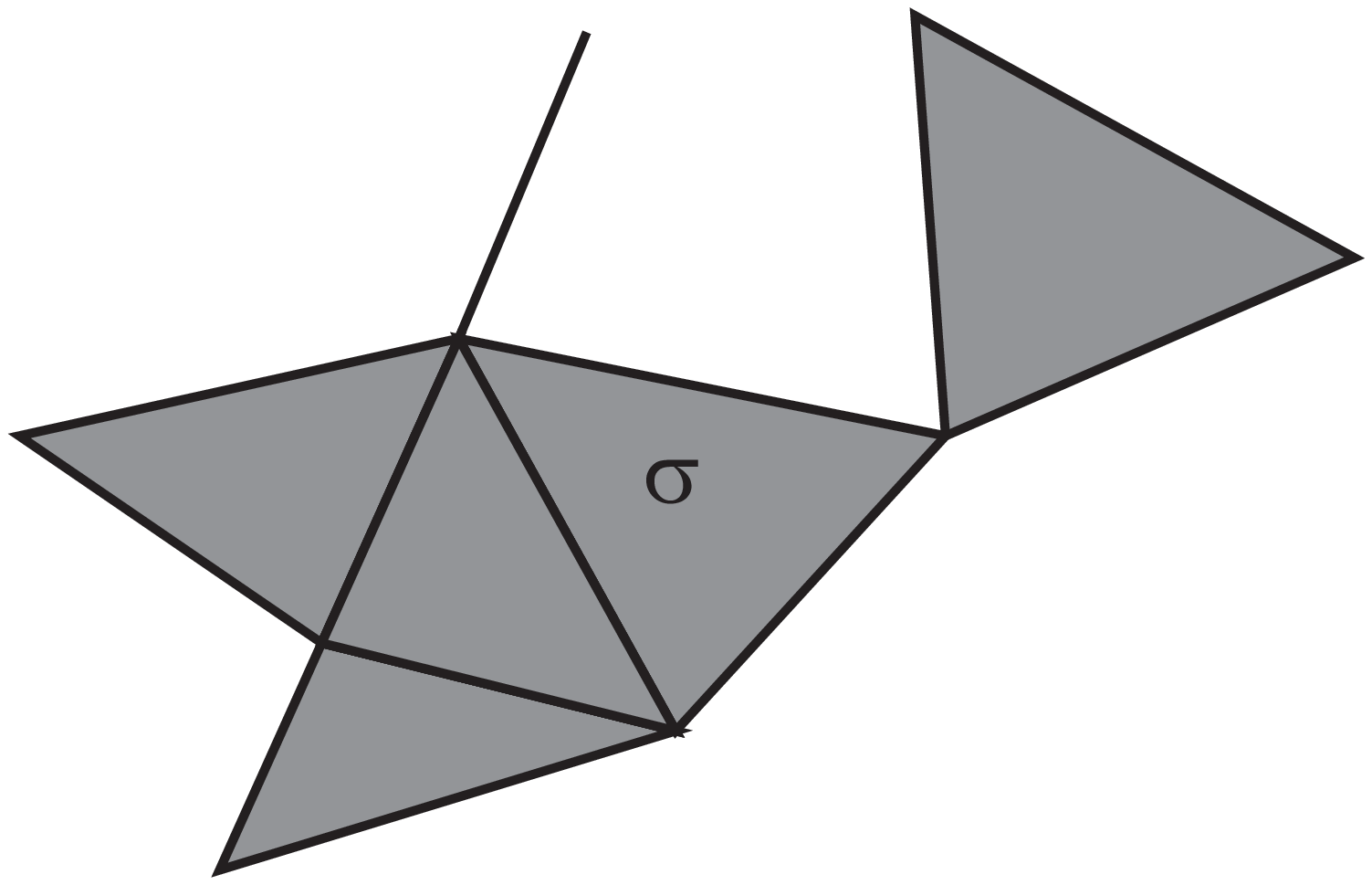}

\caption{A $2$-dimensional clique complex at the left and the star cluster of a simplex $\sigma$ at the right.}\label{starc}
\end{center}
\end{figure}

The following result is easy to prove. We give a proof for completeness.

\begin{lema} \label{dos}
Let $K$ be a complex and $K_1, K_2$ two contractible subcomplexes such that $K=K_1\cup K_2$. Then $K\simeq \Sigma (K_1\cap K_2)$. 
\end{lema}
\begin{proof}
Since $K_1$ is contractible, the inclusion $K_1\cap K_2 \hookrightarrow K_1$ extends to a map $f:v(K_1\cap K_2) \to K_1$ from the cone $v(K_1\cap K_2)$ where $v$ is a vertex not in $K$. Analogously, there is a map $g:w(K_1\cap K_2) \to K_2$ which extends the inclusion $K_1\cap K_2 \hookrightarrow K_2$. Consider the following diagram of push-outs

\begin{displaymath}
\xymatrix@C=30pt{ K_1 \cap K_2 \ar@{^{(}->}[d] \ar@{^{(}->}^i[r] & v(K_1\cap K_2) \ar@{^{(}->}[d] \ar@{->}^{\sim}_f[r] & K_1 \ar@{->}[d]\\
		  w(K_1\cap K_2) \ar@{^{(}->}^j[r] \ar@{->}^{\wr}_g[d] & \Sigma (K_1\cap K_2) \ar@{->}^-{\sim}_-{\overline{f}}[r] & K_1\cup w(K_1\cap K_2) \ar@{->}^{\wr}_{\overline{g}}[d] \\
		  K_2 \ar@{->}[rr] & & K_1\cup K_2=K. } 
\end{displaymath}
\\[0.4cm]
Since $v (K_1\cap K_2)$ is a subcomplex of $\Sigma (K_1\cap K_2)$ and $f$ is a homotopy equivalence, by the gluing theorem $\overline{f}: \Sigma (K_1\cap K_2)\to K_1\cup w(K_1\cap K_2)$ is a homotopy equivalence. The map $\overline{f}j:w(K_1\cap K_2)\hookrightarrow K_1\cup w(K_1\cap K_2)$ is also an inclusion of a subcomplex into a complex and since $g$ is a homotopy equivalence, so is $\overline{g}$. Then the composition $\overline{g}\overline{f}$ gives a homotopy equivalence from $\Sigma (K_1\cap K_2)$ to $K$.
\end{proof}

\begin{teo} \label{triangle}
Let $G$ be a graph such that there exists a vertex $v\in G$ which is contained in no triangle. Then the independence complex of $G$ has the homotopy type of a suspension. In particular, the independence complex of any triangle-free graph has the homotopy type of a suspension.
\end{teo}
This theorem will follow directly from a more refined version that we state now.

\begin{teo} \label{fuerte}
Let $G$ be a graph and let $v$ be a non-isolated vertex of $G$ which is contained in no triangle. Then $N_G(v)$ is a simplex of $I_G$ and $$I_G \simeq \Sigma (st(v) \cap SC(N_G(v))).$$ 
\end{teo}

\begin{proof}
Since $v$ is contained in no triangle, its neighbor set $N_G(v)$ is independent. Moreover, it is non-empty by hypothesis and then it is a simplex of $I_G$. By Lemma \ref{none}, $SC_{I_G}(N_G(v))$ is contractible.

If an independent set $\sigma$ of $G$ cannot be extended to $v$, then one of its vertices $w$ is adjacent to $v$ in $G$. Then $\sigma \in st(w) \subseteq SC(N_G(v))$. Therefore, $st(v)\cup SC(N_G(v))=I_G$ and the result follows from Lemma \ref{dos}.
\end{proof}

Since the reduced homology group $\widetilde{H}_n(\Sigma (K))$ is isomorphic to $\widetilde{H}_{n-1}(K)$, we deduce the following

\begin{coro}
If $K$ is the independence complex of a triangle-free graph, $H_1(K)$ is a free abelian group.
\end{coro}

The following result is due to Jonsson \cite{Jon}. His original proof makes use of discrete Morse theory. Here we exhibit a proof using star clusters.

\begin{teo} [Jonsson] \label{jonsson}
For any complex $K$ there exists a bipartite graph $G$ whose independence complex $I_G$ is homotopy equivalent to $\Sigma (K)$. 
\end{teo}
\begin{proof}
Let $V$ be the set of vertices of $K$ and let $W$ be the set of maximal simplices of $K$. Take as in \cite{Jon} the bipartite graph $G$ with parts $V$ and $W$ and whose edges are the pairs $(v,\sigma)$ with $v$ a vertex of $K$ and $\sigma$ a maximal simplex such that $v\notin \sigma$. Since $V$ and $W$ are independent, they are simplices of $I_G$, and $SC(W)$ is contractible by Lemma \ref{none}. Clearly $V\cap SC(W)=K$ and by Lemma \ref{dos}, $I_G =V \cup SC(W) \simeq \Sigma (V \cap SC(W))=\Sigma (K)$.
\end{proof}

\begin{coro}
The following homotopy classes of finite complexes coincide:

(1) Independence complexes of bipartite graphs.

(2) Independence complexes of triangle-free graphs.

(3) Independence complexes of graphs that have a vertex contained in no triangle.

(4) Suspensions of finite complexes. 
\end{coro}

Given a finitely generated abelian group $\Gamma$ and an integer $n\ge 1$, there exists a finite simplicial complex $K$ such that $H_n(K)$ is isomorphic to $\Gamma$. Then for any $n\ge 2$ and any finitely generated abelian group $\Gamma$, there exist a bipartite graph $G$ such that $H_n(I_G)$ is isomorphic to $\Gamma$. In particular, the homology groups of independence complexes of triangle-free graphs can have torsion with exception of degrees $0$ and $1$. 

\bigskip

\section{Further applications} \label{sectionapplications}

Although the notion of star cluster was introduced to be used in the proof of Theorem \ref{triangle}, we will see that it is useful to attack many problems related to independence complexes. In this section we will show some new results and we will also give alternative and shorter proofs to known results.

\subsection{A criterion for contractibility}

\begin{ej}
Let $n\ge 3$ be an odd integer. Let $G$ be the graph with vertex set $\mathbb{Z}_n\times \{a,b,c\}$ and with edge set $\{((i,a),(i,b)) \ | \ i\in \mathbb{Z}_n\}\cup \{((i,b),(i,c)) \ | \ i\in \mathbb{Z}_n\}\cup \{((i,a),(i+1,a)) \ | \ i\in \mathbb{Z}_n\}\cup \{((i,c),(i+1,c)) \ | \ i\in \mathbb{Z}_n\}$ (see Figure \ref{pentagono}).

\begin{figure}[h] 
\begin{center}
\includegraphics[scale=0.4]{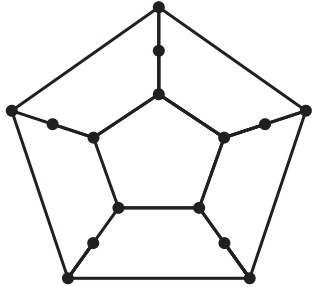}
\caption{The graph described above for $n=5$.}\label{pentagono}
\end{center}
\end{figure}

The set $\sigma= \{(i,b) \ | \ i\in \mathbb{Z}_n\}$ is independent and its star cluster coincides with the whole independence complex. Suppose that a simplex $\tau$ of $I_G$ is not in $SC(\sigma)$. Then it contains a neighbor of $(i,b)$ for every $i\in \mathbb{Z}_n$. However, $(i,b)$ and $(j,b)$ do not have common neighbors if $i\neq j$ and therefore $\tau$ has at least $n$ vertices contained in the two $n$-gons $\{((i,a),(i+1,a)) \ | \ i\in \mathbb{Z}_n\}\cup \{((i,c),(i+1,c)) \ | \ i\in \mathbb{Z}_n\}$. But then one of the two $n$-gons has $\frac{n+1}{2}$ vertices of $\tau$, and thus, two of them are adjacent in $G$. This contradicts the independence of $\tau$.
\end{ej}

The idea of the example above can be generalized in the following

\begin{obs}
Let $v_0, v_1, \ldots, v_n$ be a set of vertices in a graph $G$ such that the distance between any two of them is at least $3$. Suppose that for every collection of vertices $\{w_0,w_1, \ldots ,w_n\}$ with $w_i\in N_G(v_i)$ for every $i$, we have that there are two of them $w_i$, $w_j$, which are adjacent. Then $I_G$ is contractible.
\end{obs}

\subsection{Cycles}

In \cite{Koz}, Kozlov proved the following result

\begin{teo}[Kozlov] \label{ciclos}
Let $C_n$ be the cycle graph on $n\ge 3$ elements with vertex set $\mathbb{Z}_n$ and edges $\{(i, i+1)\ : \ i\in \mathbb{Z}_n\}$. Then the independence complex of $C_n$ is homotopy equivalent to $S^{k-1}$ if $n=3k\pm 1$ and to $S^{k-1}\vee S^{k-1}$ if $n=3k$. 
\end{teo}
\begin{proof}
Assume $n\ge 6$. Since the vertex $v=n-2$ is contained in no triangle, by Theorem \ref{fuerte}, $$I_{C_n}\simeq \Sigma(st(v)\cap SC(N(v))).$$
The simplices of $st(v)\cap SC(N(v))$ are the independent sets $\sigma$ of $C_n$ which can be extended to $n-2$ and simultaneously can be extended to $n-1$ or to $n-3$. Therefore, $st(v) \cap SC(N(v))$ is isomorphic to $I_{C_{n-3}}$. Thus, $I_{C_n}$ is homotopy equivalent to the suspension of $I_{C_{n-3}}$. The result then follows by an inductive argument analyzing the cases $n=3,4$ and $5$, which are easy to check.  
\end{proof}

In the proof we have used that suspensions of homotopy equivalent spaces have the same homotopy type. This is easy to verify. We also used that for (pointed) complexes $K$ and $L$, $\Sigma (K \vee L) \simeq \Sigma (K)\vee \Sigma (L)$. The corresponding result for reduced suspensions is trivial, and that reduced and non-reduced suspensions are homotopy equivalent follows from our first application of the gluing theorem. 

The inductive step in this proof is a particular case of Proposition \ref{subdiv4} below. 

\subsection{Forests}

A forest is a graph wich contains no cycles. In \cite[Corollary 6.1]{Ehr}, Ehrenborg and Hetyei prove that the independence complex of a forest is contractible or homotopy equivalent to a sphere. This follows from a more technical result with a long proof. In a later work \cite[Proposition 3.3]{Eng}, it is proved in a very elegant way something stronger, that such complexes are collapsible or that they collapse to the boundary of a cross-polytope. This proof relies in the following key observation (\cite[Lemma 3.2]{Eng}):

\begin{lema} [Engstr\"om] \label{engstrom}
Let $v,w$ be two different vertices of a graph $G$. If $N_G(v)\subseteq N_G(w)$, then $I_G$ collapses to $I_G \smallsetminus w$, the subcomplex of simplices not cantaining $w$.
\end{lema}

In fact the collapse of the statement is a \textit{strong collapse} in the sense of \cite{Bar}. Thus, independence complexes of forests are \textit{strong collapsible} or have the \textit{strong homotopy type} of the boundary of a cross-polytope (see \cite{Bar} for definitions). The reader interested in generalizations of Lemma \ref{engstrom} is suggested to look into \cite{Bou1} and \cite{Che}. In these papers it is proved that the simple homotopy type of the clique complex of a graph is preserved under the deletion of some vertices or edges. These results can be easily translated to independence complexes. 

Here we present a different approach to Ehrenborg and Hetyei's result using star clusters.

\begin{teo} [Ehrenborg$-$Hetyei]
The independence complex of a forest is contractible or homotopy equivalent to a sphere.
\end{teo}
\begin{proof}
Let $G$ be a forest. If $G$ is discrete, then its independence complex is a simplex or the empty set, which are contractible and a sphere, respectively. Assume then that $G$ is not discrete. Let $v$ be a leaf of $G$ and let $w$ be its unique neighbor. By Theorem \ref{fuerte}, $I_G\simeq \Sigma (st(v)\cap st(w))$. The simplices of the complex $st(v)\cap st(w)$ are exactly those independent sets of $G$ which can be extended to both $v$ and $w$. Therefore $st(v)\cap st(w)$ is the independence complex of the subgraph $H$ of $G$ induced by the vertices different from $w$ and any of its neighbors. By an inductive argument, $I_H$ has the homotopy type of a point or a sphere, and then, so does $I_G$.
\end{proof}

At this point it may seem that the complex $st(v) \cap SC(N_G(v))$ in the statement of Theorem \ref{fuerte} is always the independence complex of some graph. However this is not the case.

Consider the graph $W$ of seven vertices of Figure \ref{siete}. The vertex $v$ is contained in no triangle. The subcomplex $st(v) \cap SC(N_G(v))$ is isomorphic to the boundary of a $2$-simplex and in particular it is not clique.

\begin{figure}[h] 
\begin{center}
\includegraphics[scale=0.4]{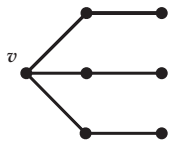}
\caption{The graph $W$.}\label{siete}
\end{center}
\end{figure}

\begin{obs} \label{inducido}
In fact, if $w$ is a non-isolated vertex of a graph $G$ which is contained in no triangle and $st(w) \cap SC(N_G(w))$ is not clique, there exists an induced subgraph $H$ of $G$ which is isomorphic to $W$ via an isomorphism $H \rightarrow W$ that maps $w$ into $v$. Moreover $st(w)\cap SC(N_G(w))$ is the independence complex of the subgraph of $G$ induced by the vertices $u\notin N_G(w)$ such that $N_G(w)\nsubseteq N_G(u)$ if there is no induced path of $G$ of length $4$ whose middle vertex is $w$.
\end{obs}

This remark can be used for instance to prove the following result. The idea of the proof is more interesting than the result itself.

\begin{prop}
If $G$ is a triangle-free graph with no induced paths of length $4$, then $I_G$ is homotopy equivalent to a sphere or it is contractible. 
\end{prop}

\subsection{Edge subdivision and Alexander dual}

Recall that the Alexander dual $K^*$ of a simplicial complex $K$ ($K$ not a simplex) with vertex set $V$ is a simplicial complex whose simplices are the proper subsets $\sigma$ of $V$ such that $V\smallsetminus \sigma \notin K$.

Let $G$ be a graph. We denote by $G'$ the subdivision of $G$ obtained when we subdivide each edge of $G$ inserting a new vertex on it.

The relationship between the independence complex of a graph and the one of its subdivision is given by the following Theorem of Csorba (\cite[Theorem 6]{Cso}).

\begin{teo} [Csorba]
Let $G$ be a non-discrete graph. Then $I_{G'}$ is homotopy equivalent to $\Sigma ((I_G) ^*)$.
\end{teo}

Csorba's proof relies on the Nerve lemma \cite[Theorem 10.6]{Bjo}, but we exhibit here an alternative simpler proof using the tools developed so far.

\begin{proof}
Let $V$ be the set of vertices of $G$. Subdividing $G$ adds a new set of vertices $W$, with one element $v_{ab}$ for each edge $(a,b)$ of $G$. Thus, the graph $G'$ is bipartite with parts $V$ and $W$. By Lemma \ref{dos}, $$I_{G'}=V \cup SC(W)\simeq \Sigma (V\cap SC(W)).$$
The simplices of $V\cap SC(W)$ are the subsets $\sigma$ of $V$ which can be extended to some vertex of $W$.  However, $\sigma \subseteq V$ can be extended to $v_{ab}\in W$ if and only if $a,b\notin \sigma$. Hence, $\sigma \subseteq V$ is in $V\cap SC(W)$ if and only if there exists an edge $(a,b)$ of $G$ such that $a,b\in V\smallsetminus \sigma$, which is equivalent to saying that $V\smallsetminus \sigma$ is not independent or, in other words, $\sigma \subsetneq V$ and $V\smallsetminus \sigma \notin I_G$. Thus $V\cap SC(W)=(I_G)^*$ and the Theorem follows.     
\end{proof}

Another result dealing with subdivisions of edges is the following (\cite[Theorem 11]{Cso})

\begin{prop} [Csorba] \label{subdiv4}
Let $G$ be a graph and $e$ an edge of $G$. Let $H$ be the graph obtained from $G$ by subdividing the edge $e$ in four parts. Then $I_H\simeq \Sigma (I_G)$.
\end{prop}
\begin{proof}
The idea is the same as in Theorem \ref{ciclos}. When the edge $e$ is replaced by a path of length $4$, three new vertices appear. The vertex $v$ in the middle of this path is contained in no triangle and $st(v)\cap SC(N_H(v))$ is isomorphic to $I_G$.
\end{proof}

From this result it is easy to compute inductively the homotopy types of independence complexes of paths (cf. \cite[Proposition 4.6]{Koz}). If $G$ is a path, $I_G$ is contractible or homotopy equivalent to a sphere.

\subsection{Homology groups of relations}

One result that is impossible not to mention when working with complexes associated to bipartite graphs, is Dowker's Theorem \cite{Dow}. Given finite sets $X$, $Y$ and a relation $\mathcal{R}\subseteq X\times Y$, two complexes are considered. The simplices of the complex $K_X$ are the non-empty subsets of $X$ which are related to a same element of $Y$. Symmetrically, the complex $K_Y$ is defined. A Theorem of C.H. Dowker \cite[Theorem 1]{Dow}, states that $K_X$ and $K_Y$ have isomorphic homology and cohomology groups. In fact they are homotopy equivalent, and moreover, simple homotopy equivalent (see \cite{Bar3}). We deduce Dowker's Theorem from our ideas of star clusters applied to bipartite graphs.

\begin{teo}[Dowker]
Let $\mathcal{R}$ be a relation between two finite sets $X$ and $Y$. Then $H_n(K_X)$ is isomorphic to $H_n(K_Y)$ and $H^n(K_X)$ is isomorphic to $H^n(K_Y)$ for every $n\ge 0$.  
\end{teo}
\begin{proof}
We may assume that $X$ and $Y$ are non-empty. Let $G$ be the bipartite graph with parts $X$ and $Y$ and where $x\in X$ is adjacent to $y\in Y$ if $x$ is not related to $y$. By Lemma \ref{dos}, $\Sigma (X\cap SC(Y))\simeq I_G \simeq \Sigma (SC(X)\cap Y)$. On the other hand it is clear that $X\cap SC(Y)=K_X$ and $SC(X)\cap Y=K_Y$. Therefore $\Sigma (K_X)$ and $\Sigma (K_Y)$ are homotopy equivalent and in particular have isomorphic homology and cohomology groups. Then, the latter is true also for $K_X$ and $K_Y$.
\end{proof}

\subsection{Kneser graph}

Let $n\ge 1$ and $k\ge 0$ be two integer numbers. The vertices of the \textit{Kneser graph} $KG_{n,k}$ are the $n$-subsets of the integer interval $\{1, \ldots , 2n+k\}$ and the edges are given by pairs of disjoint subsets. The famous Kneser conjecture formulated in 1955 by Martin Kneser states that the chromatic number of the graph $KG_{n,k}$ is $k+2$. For twenty three years this problem remained open, until L\'aszl\'o Lov\'asz managed to give finally a proof. His argument is based on a topological result known as the Lusternik-Schnirelmann Theorem. This result which involves coverings of the sphere is equivalent to the Borsuk-Ulam Theorem. As mentioned in the introduction, Lov\'asz used the neighborhood complex to turn the combinatorial data of the graph into the topological setting. A key step in his proof is Theorem \ref{lov} in Section \ref{sectionintro}, which establishes a relationship between connectivity properties of the neighborhood complex and chromatic properties of the graph. We will now study not neighborhood complexes but the topology of independence complexes of some Kneser graphs. We will explicitly compute the homotopy type of these complexes in the particular case $n=2$. In this section we will not derive any results in connection with chromatic numbers of graphs. However, the relationship between colorability and independence complexes will be analyzed in Section \ref{sectionls} of the paper, where the names of L. Lusternik and L. Schnirelmann will reappear in connection to the \textit{LS-category}, which is also related to coverings of spaces. 

The so called \textit{stable Kneser graph} $SG_{n,k}$ is the subgraph of $KG_{n,k}$ induced by the \textit{stable subsets}, i.e. subsets containing no consecutive elements (nor $1$ and $2n+k$). In \cite{Bra}, Braun studies the homotopy type of the independence complex of the stable Kneser graph for $n=2$ and proves that for $k\ge 4$, $I_{SG_{2,k}}$ is homotopy equivalent to a wedge of $2$-dimensional spheres (see \cite[Theorem 1.4]{Bra}). His proof uses discrete Morse theory. Here we show a similar result for the non-stable case.

\begin{teo} \label{kneser}
Let $k\ge 0$. Then the independence complex of $KG_{2,k}$ is homotopy equivalent to a wedge of $\binom{k+3}{3}$ spheres of dimension two.
\end{teo}
\begin{proof}
The simplices of the independence complex $I$ are given by sets of pairwise intersecting $2$-subsets of $[k+4]$. Thus, the maximal simplices of $I$ are of the form $\sigma _a=\{ \{a,b\} \}_{b\neq a}$ for some $1\le a \le k+4$ or of the form $\tau _{a,b,c}= \{ \{a,b\}, \{a,c\}, \{b,c\} \}$ for some distinct $a,b,c\in [k+4]$.
The star cluster $SC(\sigma _1)$ contains all the simplices $\sigma _a$ because $\sigma _a \in st(\{1,a\}) \subseteq SC(\sigma _1)$ for every $a\neq 1$. Moreover $\tau _{1,a,b} \in SC(\sigma _1)$ for any $a,b$. However if $a,b,c$ are different from $1$, the simplex $\tau _{a,b,c}$ is not in $SC(\sigma _1)$, although its boundary is. Therefore, $I$ is obtained from $SC(\sigma _1)$ attaching $2$-cells, one for each triple $\{a,b,c\}$ of elements different from $1$. The quotient $I / SC(\sigma _1)$ is a wedge of $\binom{k+3}{3}$ spheres of dimension two, and since $SC(\sigma _1)$ is contractible, it is homotopy equivalent to $I$.
\end{proof}

\subsection{Square grids}

Let $n,m$ be two non-negative integers. The graph $G(n,m)$ is defined as follows. The vertices are the points $(x,y)$ of the plane with integer coordinates such that $-x\le y\le x$ and $x-m\le y\le -x+n$. Two vertices are adjacent if their distance is $1$. Similarly, the vertices of the graph $H(n,m)$ are the points of $\mathbb{Z}^2$ such that $-x-1\le y\le x$ and $x-m\le y \le -x+n-1$, and again adjacent vertices correspond to points at distance $1$. It is proved in \cite[Theorem 6]{Bou2} that the homotopy type of the complexes $I_{G(n,m)}$ and $I_{H(n,m)}$ is the one of a sphere or a point. The original proof uses discrete Morse theory although there is a very simple argument based on Lemma \ref{engstrom}. This nice idea by Cukic and Engstr\"om is explained in the final remark of \cite{Bou2}. Just as another example we give an alternative proof of this result which is an application of star clusters.

Given non-negative integers $n,m$ and $k$, consider the subgraph $\widetilde{G}(n,m,k)$ of $G(n,m)$ induced by the points which satisfy $y\ge -x+k$ or $y\le x-3$. It is easy to see that $\widetilde{G}(n,m,0)=G(n,m)$ and that $\widetilde{G}(n,m+3,k)$ is isomorphic to $H(n,m)$ if $k>n$. Analogously, $\widetilde{H}(n,m,k)$ is the subgraph of $H(n,m)$ induced by the points satisfying $y \ge -x+k-1$ or $y\le x-3$. Therefore $\widetilde{H}(n,m,0)=H(n,m)$ and $\widetilde{H}(n,m+3,k)$ is isomorphic to $G(n,m)$ if $k>n$.

\begin{lema} \label{indu}
If $k<n$ and $m\neq 0$, then $$I_{\widetilde{G}(n,m,k)}\simeq \Sigma (I_{\widetilde{G}(n,m,k+3)})$$ and $$I_{\widetilde{H}(n,m,k)}\simeq \Sigma (I_{\widetilde{H}(n,m,k+3)}).$$
\end{lema}
\begin{proof}
The condition $k<n$ and $m\neq 0$ ensures that the vertex $v=([\frac{k+1}{2}],[\frac{k+1}{2}])$ is in $\widetilde{G}(n,m,k)$ and it is not isolated. Since it is contained in no triangle, $I_{\widetilde{G}(n,m,k)}\simeq \Sigma(st(v)\cap SC(N_{\widetilde{G}(n,m,k)}(v)))$.

\begin{figure}[h] 
\begin{center}
\includegraphics[scale=0.4]{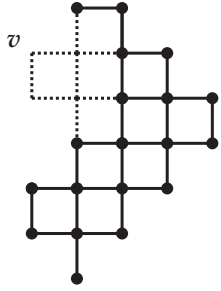}

\caption{The graph $\widetilde{G}(7,6,3)$. The complex $st(v)\cap SC(N(v))$ is the independence complex of the subgraph $\widetilde{G}(7,6,6)$ induced by the round vertices.}
\end{center}
\end{figure}

The vertex $v$ is not the middle vertex of an induced path of length $4$, therefore by Remark \ref{inducido}, $st(v)\cap SC(N_{\widetilde{G}(n,m,k)}(v))$ is the independence complex of the subgraph induced by the vertices $w$ which are not adjacent to $v$ and such that there is some neighbor of $v$ which is not adjacent to $w$. This graph is exactly $\widetilde{G}(n,m,k+3)$. The assertion for $\widetilde{H}$ follows from a similar argument.
\end{proof}  

\begin{prop}
Let $n,m,k \ge 0$. Then $I_{\widetilde{G}(n,m,k)}$ is contractible or homotopy equivalent to a sphere. The same is true for $I_{\widetilde{H}(n,m,k)}$.
\end{prop}
\begin{proof}
We prove both statements simultaneously by induction, first in $m$ and then in $n-k$. If $m=0$, $\widetilde{G}(n,m,k)$ is discrete. If it is non-empty, its independence complex is a simplex and otherwise it is a $-1$ dimensional sphere. Assume then that $m>0$. If $n-k<0$, then $\widetilde{G}(n,m,k)$ is empty when $m=1,2$ and it is isomorphic to $H(n,m-3)=\widetilde{H}(n,m-3,0)$ when $m\ge 3$. Thus the case $n-k<0$ follows by induction. Suppose then that $k\le n$.

If $k=n$, the vertex $v=([\frac{n+1}{2}],[\frac{n}{2}])\in \widetilde{G}(n,m,k)$ is isolated, and then the independence complex is contractible. We can assume that $m,n$ and $k$ satisfy the hypothesis of Lemma \ref{indu}. Therefore $I_{\widetilde{G}(n,m,k)}\simeq \Sigma (I_{\widetilde{G}(n,m,k+3)})$ and by induction, it has the homotopy type of a sphere or a point.

Similarly the same is true for $I_{\widetilde{H}(n,m,k)}$.
\end{proof}

In particular when $k=0$ we obtain the result of \cite{Bou2}.

\begin{coro} [Bousquet-M\'elou$-$Linusson$-$Nevo]
The independence complexes of the graphs $G(n,m)$ and $H(n,m)$ are contractible or homotopy equivalent to a sphere.
\end{coro}

\subsection{Order complexes}

The order complex of a finite poset $P$ is the simplicial complex whose simplices are the non-empty chains of $P$. Order complexes are clique and therefore it is possible to use our results to study them.

\begin{ej}
The order complex of the poset whose Hasse diagram is in Figure \ref{hasse1} is contractible.
\begin{figure}[h] 
\begin{center}
\includegraphics[scale=0.4]{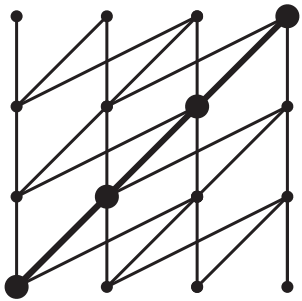}

\caption{A poset with a chain that intersects all the maximal chains.}\label{hasse1}
\end{center}
\end{figure}
The points of the diagonal constitute a chain $\sigma$ which intersects any maximal chain of the poset. In other words, the order complex of $P$ is the star cluster of $\sigma$.
\end{ej}

The following result summarizes the idea of the example.

\begin{prop}
Let $P$ be a finite poset. Suppose that there exists a chain of $P$ which intersects any maximal chain. Then the order complex of $P$ is contractible.
\end{prop}

In fact it can be proved that in the hypothesis of the proposition, the poset $P$ is dismantlable, or equivalently that the order complex is strong collapsible, which is something stronger (see \cite{Bar}). Order complexes appear in problems of different areas of mathematics, like algebraic topology, combinatorics, group theory and discrete geometry. They allow to establish the relationship between the homotopy theory of simplicial complexes and of finite topological spaces \cite{Bar2}. 
  
\bigskip

\section{Matching complexes and claw-free graphs} \label{sectionmatchings}

The paper \cite{Ehr}, that we have already discussed, concludes with a question about the homotopy type of a polyhedron called \textit{Stirling complex}.  

\begin{defi}
Let $n\ge 2$. The vertices of the \textit{Stirling complex} $Stir _n$ are the pairs $(i,j)$ with $1\le i<j \le n$. The simplices are given by sets of vertices which pairwise differ in the first and in the second coordinate. 
\end{defi}

The number of $k$-dimensional simplices of $Stir _n$ is the Stirling number of second kind $S(n,n-k-1)$ (see \cite[Proposition 2.4.2]{Sta}). In our attempt to attack this problem, we will prove a general result on the connectivity of a certain class of well-known complexes.

Given a graph $G$, its \textit{matching complex} $M(G)$ is defined as the simplicial complex whose simplices are the non-trivial matchings of $G$, that is, the non-empty collections of edges which are pairwise non-adjacent. It is easy to see that matching complexes are independence complexes of graphs. Precisely, $M(G)=I_{\mathfrak{E}(G)}$, where $\mathfrak{E}(G)$ denotes the \textit{edge graph} (or \textit{line graph}) of $G$. The vertices of $\mathfrak{E}(G)$ are the edges of $G$ and its edges are given by adjacent edges of $G$. In the last twenty years, two (classes of) matching complexes were particularly studied. One is the matching complex $M_n$ of the complete graph $K_n$. The other, known as the \textit{chessboard complex} $M_{n,m}$, is the independence complex of the complete bipartite graph $K_{n,m}$. The vertices of the chessboard complex $M_{n,m}$ can be considered as the squares of an $n\times m$ chessboard and its simplices as the sets of squares which can be occupied by rooks in such a way that no rook attacks another. Note that the Stirling complex $Stir _n$ also is a matching complex and it is closely related to $M_{n,m}$. The difference is that in $Stir _n$ the rooks are only allowed to be above the diagonal.

Some of the most important results obtained in relation to the homotopy of the spaces $M_n$ and $M_{n,m}$ are about connectivity and existence of torsion in homology groups. Bounds for the connectivity were proved by Bj\"orner, Lov\'asz, Vre\'cica, $\check{\textrm{Z}}$ivaljevi\'c \cite{Bjo2} and Bouc \cite{Bou3}.

\begin{teo} [Bj\"orner$-$Lov\'asz$-$Vre\'cica$- \check{\textrm{Z}}$ivaljevi\'c, Bouc] \label{sharp}
Let $n,m$ be positive integers. Then $M_n$ is $[\frac{n-5}{3}]$-connected and $M_{n,m}$ is $min\{n-2,m-2, \frac{n+m-5}{3}\}$-connected. 
\end{teo}

It was conjectured in \cite{Bjo2} that the bounds given by this result are in fact optimal. Some cases were first established by Bouc \cite{Bou3} but the complete result was obtained by Shareshian and Wachs \cite{Sha}.    

In this section we will prove a general bound for the connectivity of a matching complex. In fact we will show a stronger result, regarding independence complexes of \textit{claw-free graphs}. In particular we will apply this to study the Stirling complex. We will also use these ideas to prove two results on the connectivity of independence complexes of general graphs.

\begin{defi}
A graph is called \textit{claw-free} if it has no induced subgraph isomorphic to the complete bipartite graph $K_{1,3}$. 
\end{defi}

The following result is due to Engstr\"om \cite[Theorem 3.2]{Eng2}.

\begin{teo} [Engstr\"om] \label{engclaw}
Let $G$ be a claw-free graph with $n$ vertices and maximum degree $m$. Then $I_G$ is $[\frac{2n-1}{3m+2}-1]$-connected.
\end{teo}

The maximum degree of a graph is the maximum degree of a vertex of the graph. This improves a similar result for general graphs which says that $I_G$ is $[\frac{n-1}{2m}-1]$-connected (see \cite{Eng}). More results about connectivity and homotopy of independence complexes of graphs with bounded maximum degree will be given in Section \ref{sectionsuspension}. 
The main result of this section is the following bound for the connectivity of independence complexes of claw-free graphs which does not depend on the maximum degree $m$.

\begin{teo} \label{conn}
Let $G$ be a claw-free graph. Then $I_G$ is $[\frac{dim(I_G)-2}{2}]$-connected.
\end{teo}
\begin{proof}
Let $\sigma$ be an independent set of $G$ of maximum cardinality $d+1=dim(I_G)+1$. Suppose that $\tau\in I_G$ and $r=dim(\tau)\le [\frac{d-2}{2}]$. Since $G$ is claw-free and $\sigma$ is independent, every vertex of $\tau$ is adjacent to at most two vertices of $\sigma$. Since $2(r+1)\le d<d+1$, there is a vertex of $\sigma$ which is not adjacent to any vertex of $\tau$. Therefore the independent set $\tau$ can be extended to some vertex of $\sigma$, which means that $\tau \in SC(\sigma)$.
Since $SC(\sigma)$ contains the $[\frac{d-2}{2}]$-skeleton of $I_G$, the relative homotopy groups $\pi _k (I_G, SC(\sigma))$ are trivial for $k\le [\frac{d-2}{2}]$. The result now follows from the contractibility of $SC(\sigma)$ and the long exact sequence of homotopy groups of the pair $(I_G, SC(\sigma))$, $$ \ldots \rightarrow \pi _k (SC(\sigma)) \rightarrow \pi _k (I_G) \rightarrow \pi _k (I_G, SC(\sigma)) \rightarrow \ldots $$       
\end{proof}

This bound and Theorem \ref{engclaw} give different information. In same cases the number $[\frac{2n-1}{3m+2}-1]$ is smaller than $[\frac{dim(I_G)-2}{2}]$ and in others it is bigger. 

Since the homology of a complex is trivial for degrees greater than its dimension and since the homology groups of degree less than or equal to $k$ are trivial for $k$-connected complexes by the Hurewicz theorem, from Theorem \ref{conn} we deduce the following

\begin{coro}
Let $G$ be a claw-free graph. Then there exists an integer $n$ such that the support of the reduced homology of $I_G$ lies in the interval $\{[\frac{n}{2}], \ldots ,n\}$.
\end{coro}

\begin{ej}
The Independence complex of the claw-free graph in Figure \ref{s1s2} is homotopy equivalent to $S^1\vee S^2$. Therefore, the support of its reduced homology is exactly the interval $\{[\frac{2}{2}],\ldots ,2\}$.

\begin{figure}[h] 
\begin{center}
\includegraphics[scale=0.4]{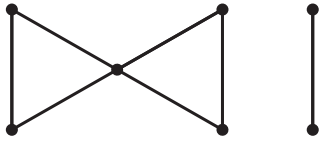}
\caption{A claw-free graph $G$ such that $\widetilde{H}_k(I_G)\neq 0$ only for $k=1$ and $k=2$.}\label{s1s2}
\end{center}
\end{figure}

\end{ej}

The following observation shows that there is a direct relationship between matching complexes and claw-free graphs.

\begin{obs}
If three edges in a graph $G$ are adjacent to a same other edge, then two of the first are adjacent. Therefore for every graph $G$, the edge graph $\mathfrak{E}(G)$ is claw-free. In particular the matching complex $M(G)$ is the independence complex of a claw-free graph.
\end{obs}

\begin{ej}
The graph $G$ of Figure \ref{notmatch} is claw-free but it is not isomorphic to $\mathfrak{E}(H)$ for any graph $H$. Therefore $I_G$ is the independence complex of a claw-free graph which is not isomorphic to any matching complex, for if $I_G$ were isomorphic to $M(H)=I_{\mathfrak{E}(H)}$, then the complementary graphs of their $1$-skeletons, $G$ and $\mathfrak{E}(H)$ should be isomorphic graphs.

\begin{figure}[h]
\begin{center}
\includegraphics[scale=0.4]{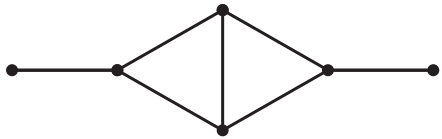}
\caption{A claw-free graph whose independence complex is not a matching complex.}\label{notmatch}
\end{center}
\end{figure}

\end{ej}

\begin{coro} \label{conn2}
If $K$ is a matching complex, then it is $[\frac{dim(K)-2}{2}]$-connected. In particular there is an integer $n$ such that the support of $\widetilde{H}(K)$ lies in $\{[\frac{n}{2}], \ldots ,n\}$.
\end{coro}

Since the dimension of $Stir _n$ is $n-2$, we deduce the following

\begin{coro}
The Stirling complex $Stir _n$ is $[\frac{n-4}{2}]$-connected.
\end{coro}

Analogous results could be deduced for the complexes $M_n$ and chessboard complexes. Corollary \ref{conn2} says that $M_n$ is $$[\frac{dim(M_n)-2}{2}]=[\frac{1}{2}([\frac{n}{2}]-1-2)]=[\frac{1}{2}[\frac{n-6}{2}]]=[\frac{n-6}{4}]\textrm{-connected,}$$ 
and that $M_{n,m}$ is $$[\frac{dim(M_{n,m})-2}{2}]=[\frac{n-3}{2}]\textrm{-connected}$$ if $n\le m$. Therefore the bounds given by Theorems \ref{conn} and \ref{conn2} are not optimal in these particular examples in contrast with the bounds of Theorem \ref{sharp}. However, we prove that in the generality of Theorems \ref{conn} and \ref{conn2}, the bounds we exhibit are sharp in the following sense:

\begin{prop}
For every non-negative integer $n$ there exists a matching complex $K$ (and therefore an independence complex of a claw-free graph) such that $K$ is $n$-dimensional and it is not $[\frac{n}{2}]$-connected. 
\end{prop}
\begin{proof}
Given $k\ge 1$ consider the graph $A_k$ which is a disjoint union of $k$ squares and the graph $B_k$ which is the disjoint union of two adjacent edges and $k-1$ squares (see Figure \ref{cuadrados}).

\begin{figure}[h] 
\begin{center}
\includegraphics[scale=0.35]{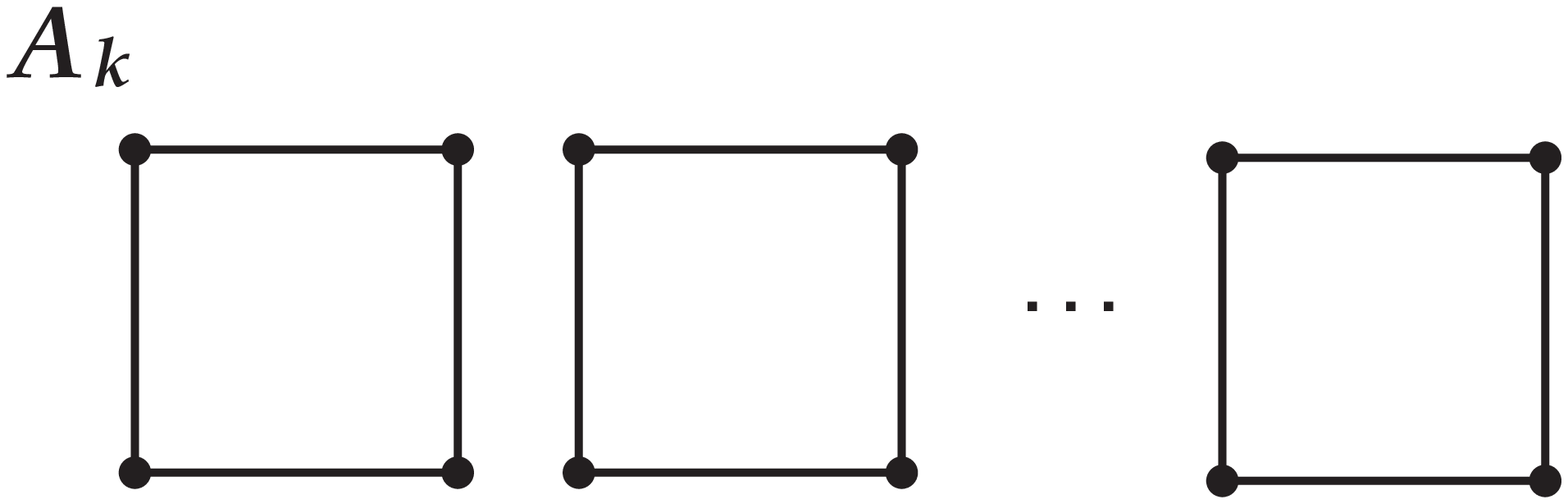}

\includegraphics[scale=0.35]{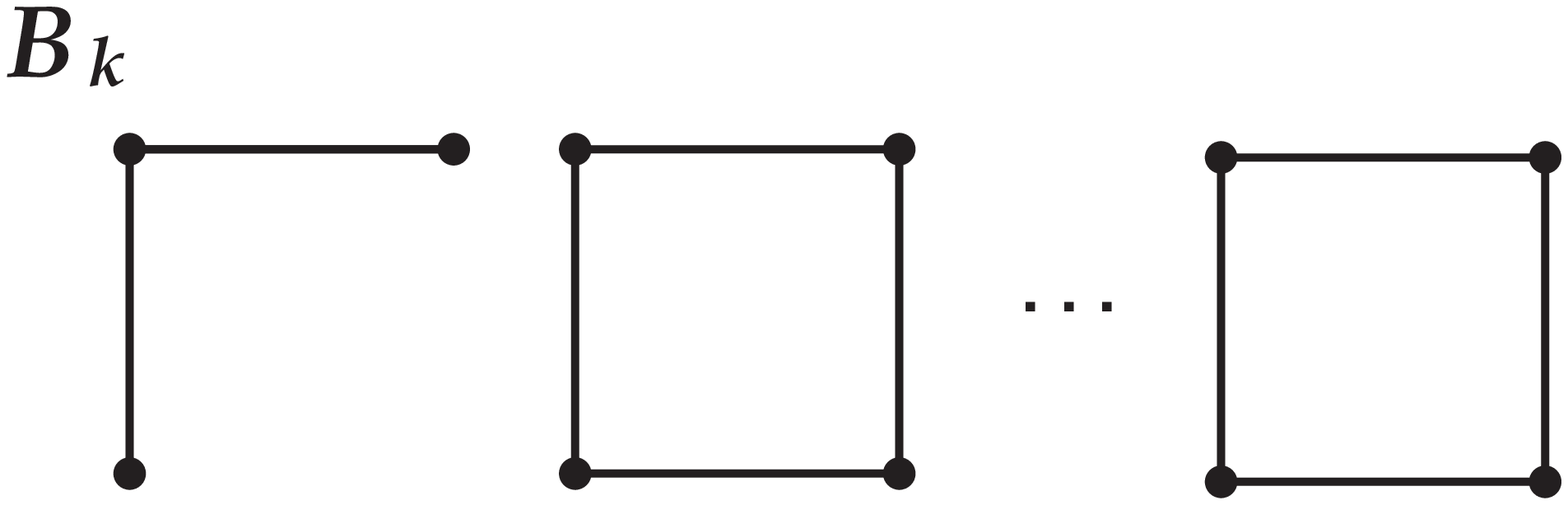}
\caption{Graphs $A_k$ and $B_k$.}\label{cuadrados}
\end{center}
\end{figure}

The complex $M(A_k)$ is the independence complex of $\mathfrak{E}(A_k)= A_k$, which is the join of the independence complexes of the connected components of $A_k$. Since the independence complex of a square is homotopy equivalent to $S^0$, $M(A_k)\simeq (S^0)^{*k}=S^{k-1}$ is not $(k-1)$-connected. On the other hand $dim(M(A_k))=2k-1$. Therefore we have constructed a matching complex which is $(2k-1)$-dimensional but not $[\frac{2k-1}{2}]=(k-1)$-connected.

Similarly, $M(B_k)$ is the independence complex of $\mathfrak{E}(B_k)$, which is the disjoint union of an edge and $k-1$ squares. Therefore $M(B_k)\simeq \Sigma ((S^0)^{*(k-1)})=S^{k-1}$. Thus $M(B_k)$ is $(2k-2)$-dimensional but not $[\frac{2k-2}{2}]=(k-1)$-connected.

Clearly the numbers $2k-1$ and $2k-2$ take all the non-negative integer values for $k\ge 1$.
\end{proof}

A difficult problem seems to be the classification of all the homotopy types of matchings complexes and of independence complexes of claw-free graphs or, at least to see whether these two sets differ.

\bigskip

Theorem \ref{conn} can be deduced from the following result.

\begin{prop} \label{mascon}
If $\sigma$ is an independent set of a graph $G$ such that every independent set $\tau$ of $G$ with at most $r$ vertices can be extended to some vertex of $\sigma$, then $I_G$ is $(r-1)$-connected.
\end{prop}

The proof can be made as for Theorem \ref{conn} using the long exact sequence of homotopy groups of the pair $(I_G, SC(\sigma))$ (or collapsing $SC(\sigma)$ and using cellular approximation). This result appears mentioned in \cite{Aha} where the notion of \textit{bi-independent domination number} of a graph is introduced. In \cite{Aha} other general bounds for the connectivity of independence complexes are stated in connection with different notions of domination numbers. The proof of this particular result is really simple using star clusters.  
A useful application of Proposition \ref{mascon} to general graphs is the next

\begin{coro} \label{lejos}
Let $G$ be a graph. Let $S$ be a subset of vertices of $G$ which satisfies that the distance between any two elements of $S$ is at least $3$. Then $I_G$ is $(\#S-2)$-connected.  
\end{coro}

The last result of this section relates the connectivity of the independence complex of a general graph $G$ with the diameter of $G$.

\begin{coro}
Let $G$ be a connected graph of diameter $n$. Then $I_G$ is $[\frac{n}{3}-1]$-connected. 
\end{coro}
\begin{proof}
Let $v,w\in G$ such that $d(v,w)=n$ and let $v=v_0, v_1, \ldots , v_n=w$ be a path in $G$. Then the set $S=\{v_0, v_3, \ldots , v_{3[\frac{n}{3}]}\}$ satisfies the hypothesis of Corollary \ref{lejos}. 
\end{proof}

\bigskip

\section{Chromatic number and Strong Lusternik-Schnirelmann category} \label{sectionls}

In this section we present a new approach to study the relationship between the chromatic number of a graph and the topology of the associated complex. The \textit{category} of a topological space is a numerical homotopy invariant that was introduced by L. Lusternik and L. Schnirelmann in the thirties. This deeply studied notion is closely related to other well known concepts such as the \textit{cup length} (the maximum number of positive degree elements in the cohomology ring whose product is non-trivial); the minimum number of critical points of a real valued map, when the space is a manifold; the homotopical dimension of the space. The category of a space in general is not easy to determine. The reader is referred to \cite{Cor2} for results on this invariant.

\begin{defi}
Let $X$ be a topological space. The \textit{strong (Lusternik-Schnirelmann) category} $Cat(X)$ of $X$ is the minimum integer number $n$ such that there exists a CW-complex $Y$ homotopy equivalent to $X$ which can be covered by $n+1$ contractible subcomplexes. If there is not such an integer, we say that $Cat(X)$ is infinite.
\end{defi}

For instance, a space has strong category $0$ if and only if it is contractible. Lemma \ref{dos} is still true if we consider not necessarily finite CW-complexes and therefore a space has strong category less than or equal to $1$ if and only if it has the homotopy type of a suspension. The $2$-dimensional torus $S^1\times S^1$ is an example of space with strong category equal to $2$. There are spaces with arbitrarily large strong category. The following results establish a direct relationship between the chromatic number of a graph and the strong category of its independence complex.

In the next, $N_G(v)$ denotes the subgraph of $G$ induced by the neighbors of $v$.

\begin{teo} \label{catloc}
Let $v$ be a vertex in a graph $G$. Then $$Cat(I_G)\le \chi (N_G(v)).$$
\end{teo}
\begin{proof}
If $n=\chi (N_G(v))$, then the set of neighbors of $v$ can be partitioned into $n$ independent sets $\sigma _1, \sigma _2, \ldots , \sigma _n$. The contractible subcomplexes $st(v), SC(\sigma _1), SC(\sigma _2), \ldots ,SC(\sigma _n)$ cover $I_G$ since an independent set which cannot be extended to $v$ must contain a neighbor of $v$. Thus, $Cat(I_G)\le n$. 
\end{proof}


Theorem \ref{catloc} says that $Cat(I_G)\le min \{\chi (N_G(v)) \ | \ v\in G \}$. On the other hand it is clear that $\chi (G) \ge max \{\chi (N_G(v)) \ | \ v\in G \}+1$. In particular we deduce the following

\begin{coro} \label{cat}
Let $G$ be a graph. Then $\chi (G) \ge Cat(I_G)+1$. 
\end{coro}

An alternative proof of Corollary \ref{cat} is as follows. Let $n=\chi (G)$. Then the set of vertices of $G$ can be partitioned into $n$ independent sets, and clearly their star clusters cover $I_G$. 

\begin{ej} \label{colores}
The chromatic number of the graph $G$ of Figure \ref{color} is $4$. However the bound of Corollary \ref{cat} is not sharp since there is a vertex $v$ such that $\chi (N_G(v))=2$ and then by Theorem \ref{catloc}, $Cat(I_G)\le 2$. In this case, the equality holds.
\begin{figure}[h] 
\begin{center}
\includegraphics[scale=0.5]{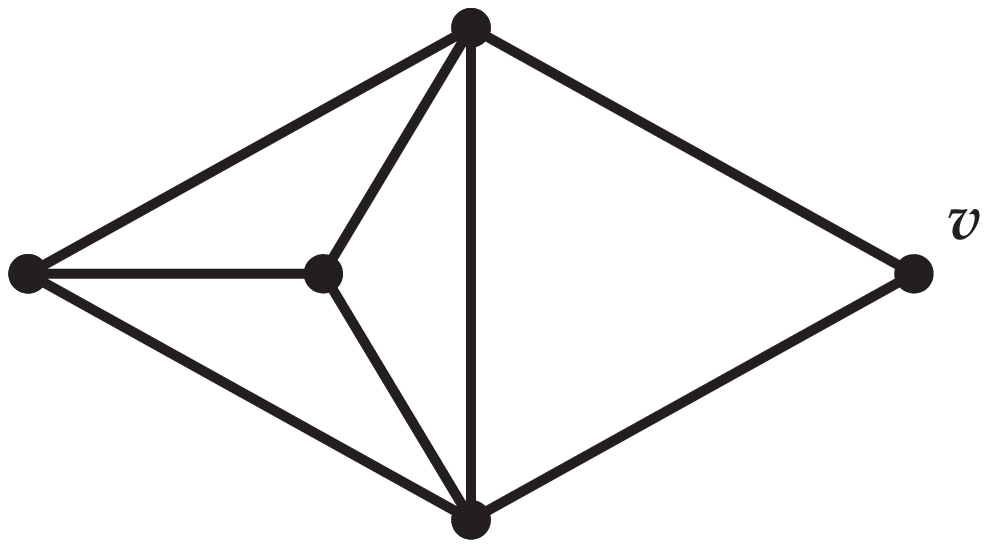}
\caption{The graph $G$ in Example \ref{colores}.}\label{color}
\end{center}
\end{figure}
\end{ej}

Unfortunately, the numbers $Cat(I_G)$ and $\chi (N_G(v))$ of Theorem \ref{catloc} can be very far one from the other. Consider the Kneser graph $G=KG_{2,k}$ for $k\ge 2$. Then the subgraph induced by the neighbors of any vertex $v=\{a,b\}$ is isomorphic to $KG_{2,k-2}$. By the Theorem of Lov\'asz, the Kneser conjecture, $\chi (N_G(v))=k$. On the other hand $Cat(I_G)=1$ since $I_G$ is homotopy equivalent to a wedge of $2$-dimensional spheres by Corollary \ref{kneser}.

\begin{ej}
It is well known that the cup-length of a space is a lower bound for the strong category. The cup length of the complex projective space $\mathbb{CP}^n$ equals its complex dimension $n$. In particular we deduce from Theorem \ref{cat} that if $G$ is a graph whose independence complex is homotopy equivalent to $\mathbb{CP}^n$, then $\chi (G)\ge n+1$.

The cup length of the $n$-dimensional torus $T^n$ is also $n$. Therefore if $I_G\simeq T^n$, $\chi (G)\ge n+1$.
\end{ej}

\begin{coro} \label{planarcat}
If $G$ is a planar graph, then $Cat(I_G)\le 3$.
\end{coro}
\begin{proof}
This follows immediately from the Four colour theorem and Corollary \ref{cat}, but we give a proof using more basic results. Since $G$ is planar, there exists a vertex $v$ of degree less than or equal to five. Again by the planarity of $G$, the subgraph $N_G(v)$ induced by the neighbors of $v$ does not contain a clique of four vertices. Then it is easy to check that $\chi (N_G(v))\le 3$ and the result follows from Theorem \ref{catloc}.  
\end{proof}



\bigskip

\section{Suspensions of graphs} \label{sectionsuspension}

We introduce a construction that will allow us to generalize some results mentioned in the previous sections and a new result of Skwarski \cite{Skw} on independence complexes of planar graphs. We will also use this to prove some results on graphs with bounded maximum degree.


\begin{defi}
Let $G$ be a graph and $H\subseteq G$ a subgraph. The \textit{suspension} of $G$ over $H$ is a graph $S(G,H)$ whose vertices are those of $G$, a new vertex $v$, and a new vertex $v_M$ for each maximal independent set of vertices $M$ in the graph $H$. The edges of $S(G,H)$ are those edges of $G$ which are not edges of $H$, the pairs $(v,v_M)$ and the pairs $(w,v_M)$ with $w$ a vertex of $H$ which is not contained in $M$. 
\end{defi}

\begin{prop} \label{suspension}
Let $H$ be a subgraph of a graph $G$. Then the independence complex of $S(G,H)$ is homotopy equivalent to $\Sigma(I_G)$.
\end{prop}
\begin{proof}
Since $v$ is contained in no triangle, by Theorem \ref{fuerte} it is enough to prove that $st_{S(G,H)}(v)\cap SC(N_{S(G,H)}(v))=I_G$. Let $\sigma \in st(v)\cap SC(N(v))$ and let $w,w'\in \sigma$. It is clear that $w,w'\in G$. Suppose that $w,w'\in H$. The fact that $\{w,w'\}\in SC(N(v))$ says that there exists a maximal independent set $M$ of $H$ such that $\{w,w'\}\subseteq M$ and, in particular, $(w,w')$ is not an edge of $H$. Since $\sigma$ is independent in $S(G,H)$, $(w,w')$ is not an edge of $G$ either. Since this is true for any pair of vertices $w,w'\in \sigma$, $\sigma \in I_G$. Conversely, if $\sigma \neq \emptyset$ is an independent set of vertices of $G$, it is also independent in $S(G,H)$. Moreover $\sigma \cap H$ is independent in $H$ and then there exists a maximal independent set $M$ of $H$ containing $\sigma \cap H$. Thus, $\sigma$ can be extended to $v_M$. Therefore $\sigma \in st(v) \cap SC(N(v))$.
\end{proof}

If the subgraph $H$ is discrete, then $S(G,H)$ is just the disjoint union of $G$ and an edge. If $H$ is an edge, then $S(G,H)$ is the subdivision described in Proposition \ref{subdiv4}, which consist in replacing the edge by a path of length $4$. If $H=G$ then $S(G,H)$ is a bipartite graph. Taking a graph $G$ such that $I_G$ is the barycentric subdivision of a complex $K$, this provides an alternative proof of Jonsson's Theorem \ref{jonsson} which says that for any complex $K$ there exists a bipartite graph whose independence complex is homotopy equivalent to $\Sigma (K)$. An interesting application of Proposition \ref{suspension} is the following alternative proof of a result recently proved by Skwarski \cite{Skw}.  

\begin{teo}[Skwarski] \label{skwar}
Let $K$ be a complex. Then there exists a planar graph whose independence complex is homotopy equivalent to an iterated suspension of $K$.
\end{teo}
\begin{proof}
By Proposition \ref{suspension} it is enough to show that for any graph $G$ there is a sequence $G=G_0, G_1, \ldots, G_n$ of graphs in which $G_{i+1}$ is the suspension of $G_i$ over some subgraph $H_i$, and such that $G_n$ is planar. We will show that this is possible taking the subgraphs $H_i\subseteq G_i$ as just one edge or two non-adjacent edges.

We consider the graph $G$ drawn in the plane, in such a way that the vertices are in general position, all the edges are straight lines and no three of them are concurrent.

Suppose that only two edges $(w_0,w_1)$, $(z_0,z_1)$ of $G$ intersect. In this case we take the subgraph $H$ as the union of these edges. The suspension $S(G,H)$ is a planar graph. Indeed, it is possible to find an embedding of $S(G,H)$ in the plane keeping all the vertices of $G$ fixed, choosing $v$ as the intersection of $(w_0,w_1)$ and $(z_0,z_1)$, and each of the other four new vertices $v_{\{w_i,z_j\}}$ ($i,j\in \mathbb{Z}_2$) in the triangle $w_{i+1}vz_{j+1}$ closely enough to $v$ (see Figure \ref{sus}).

\begin{figure}[h] 
\begin{center}
\includegraphics[scale=0.4]{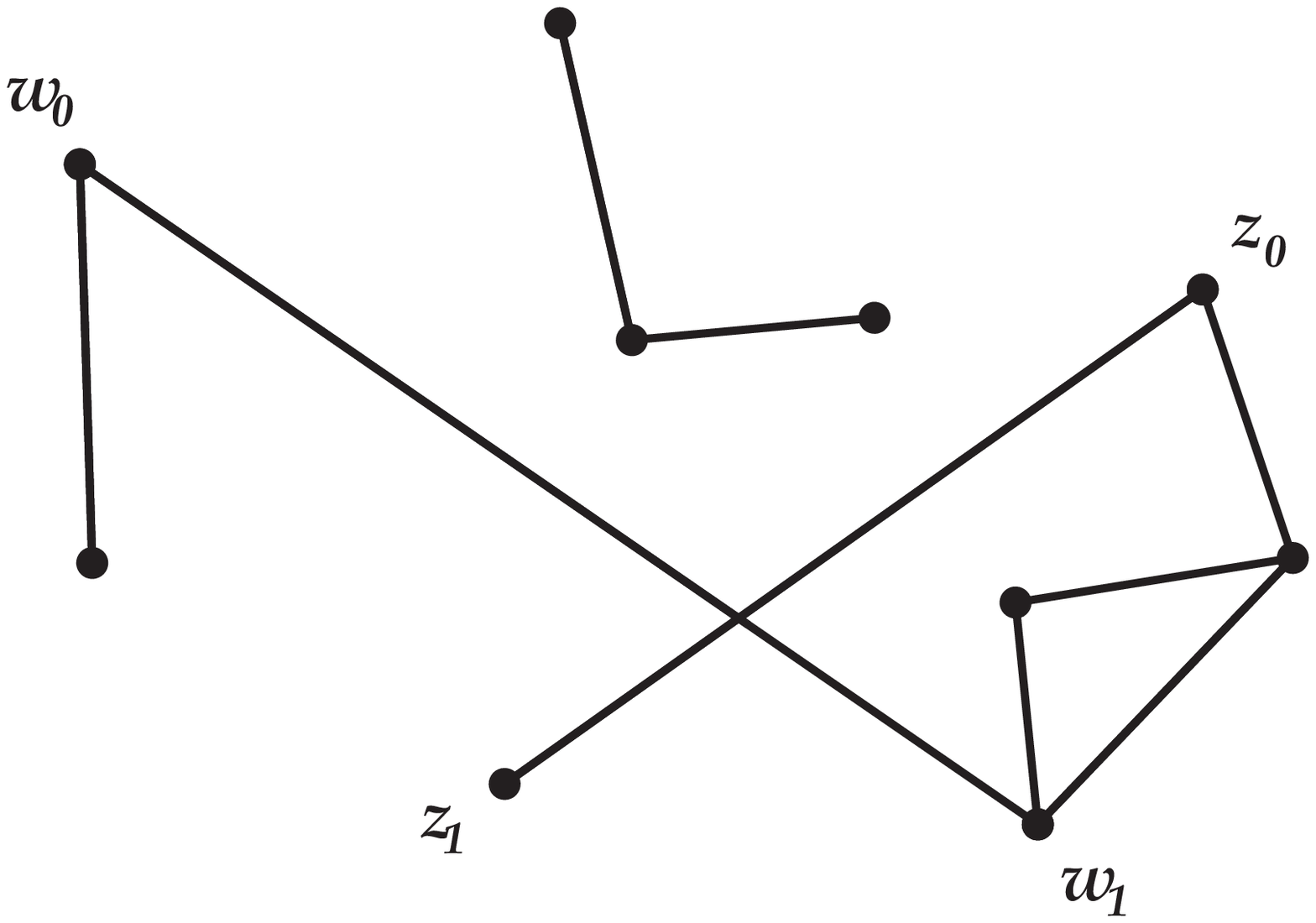}
\qquad
\qquad
\includegraphics[scale=0.4]{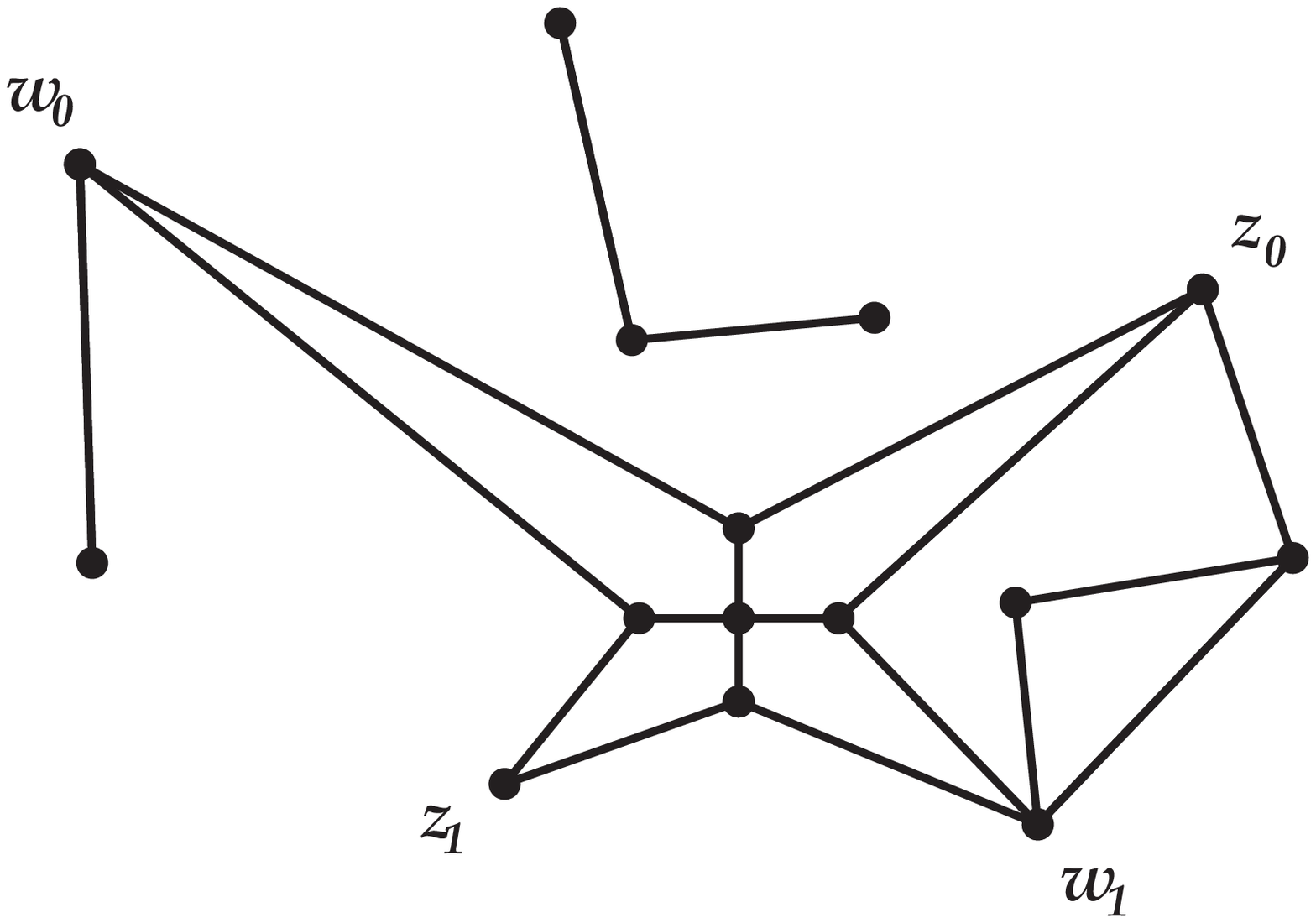}
\caption{In the graph at the left two edges intersect. At the right, a suspension of the graph embedded in the plane.}\label{sus}
\end{center}
\end{figure}

Now suppose that in the general case $G$ has an arbitrary number of intersecting edges. First we subdivide the edges, making suspensions over one edge at the time, in such a way that the resulting graph is drawn in the plane with each edge intersecting at most one other edge. Then we take suspensions one at the time over each pair of intersecting edges, following the idea of the basic case described above. In each step the number of intersecting edges is reduced by one, and finally a planar graph is obtained.
\end{proof}

This result says that the homology groups of planar graphs can be arbitrarily complicated. On the other hand Corollary \ref{planarcat} asserts that not any homotopy type is realized as the independence complex of a planar graph. It remains as an open problem the description of the homotopy types of independence complexes of planar graphs. 

\bigskip

Recall once again that the maximum degree of a graph $G$ is the maximum number $m$ among all the degrees of the vertices of $G$. If $m=1$, $I_G$ is homotopy equivalent to a sphere or it is contractible. If $m\le 2$, $G$ is a disjoint union of cycles and paths, and by the results of Section \ref{sectionapplications}, $I_G$ is contractible or homotopy equivalent to a join of discrete spaces of cardinalities $2$ or $3$. Since for complexes $K, L$, we have that $K*L\simeq \Sigma (K\wedge L)$, where $K\wedge L$ denotes the smash product between $K$ and $L$, it is easy to see that the join between $\bigvee\limits_{i=1}^{r} S^n$ and the discrete space of $k$ points is homotopy equivalent to $\bigvee\limits_{i=1}^{r(k-1)} S^{n+1}$. Therefore by an inductive argument it is possible to prove the following

\begin{prop}
A complex $K$ is homotopy equivalent to the independence complex of a graph with maximum degree $m\le 2$ if and only if $K$ is contractible or homotopy equivalent to $\bigvee\limits_{i=1}^{2^r} S^n$ for some $n\ge r-1$. 
\end{prop}

The next case, for graphs with maximum degree less than or equal to $3$ is already much more complex. We will use suspensions of graphs to prove a result similar to Theorem \ref{skwar} for that class.  

\begin{teo} \label{maxdeg3}
Let $K$ be a complex. Then there exists a graph of maximum degree not greater than $3$ whose independence complex is homotopy equivalent to an iterated suspension $\Sigma ^r(K)$ of $K$.
\end{teo}
\begin{proof}
As in the proof of Theorem \ref{skwar}, we will see that for any graph $G$ there exists a sequence of graphs, each of which is the suspension of the previous one over some subgraph, that starts in $G$ and ends in a graph with maximum degree less than or equal to $3$. We make the proof by induction (in the sum of the degrees of the vertices $v\in G$ such that $deg(v)>3$).

Suppose that $G$ contains a vertex $w$ of degree greater than $3$. Let $w_1$ and $w_2$ be two different neighbors of $w$. Consider the subgraph $H$ which consist of the vertices $w,w_1,w_2$ and the edges $(w,w_1),(w,w_2)$. The suspension $S(G,H)$ contains the vertices of $G$ and three new vertices $v, v_{\{w\}}$ and $v_{\{w_1,w_2\}}$. The degrees in $S(G,H)$ of the vertices different from $w, v,  v_{\{w\}}$ and $v_{\{w_1,w_2\}}$ are the same as in $G$. But $deg_{S(G,H)}(w)=deg_G(w)-1$, $deg(v)=2$, $deg(v_{\{w\}})=3$ and $deg(v_{\{w_1,w_2\}})=2$. By induction there exists a sequence as we want starting in $S(G,H)$, and therefore there is one starting in $G$.    
\end{proof}

As we mentioned in Section \ref{sectionmatchings}, the independence complex of a graph with $n$ vertices and maximum degree $m$ is $[\frac{n-1}{2m}-1]$-connected. The ideas used in that section lead us to a similar result but involving the dimension of the complex instead of the number of vertices. 

\begin{prop} \label{maxdegm}
Let $G$ be a graph of maximum degree $m$. Then $I_G$ is $[\frac{dim(I_G)}{m}-1]$-connected.
\end{prop}
\begin{proof}
Let $\sigma$ be an independent set of $G$ of maximum cardinality $d+1=dim(I_G)+1$. If another independent set $\tau$ has at most $[\frac{d}{m}]$ vertices, it can be extended to some vertex of $\sigma$. The result follows then from Proposition \ref{mascon}. 
\end{proof}

We can deduce then that if $G$ has maximum degree $m$, the support of the homology of $I_G$ lies in an interval of the form $\{[\frac{n}{m}], \ldots , n\}$. Therefore, the independence complexes of graphs with maximum degree $3$ do not cover all the homotopy types of complexes. Moreover, the previous remark gives a lower bound for the number $r$ in the statement of Theorem \ref{maxdeg3}. If $\widetilde{H}_i(K)\neq 0 \neq \widetilde{H}_j(K)$ for some $i<j$, then $r$ has to be greater than $\frac{j-3i-3}{2}$. If $\Sigma ^r(K)\simeq I_G$ for $G$ of maximum degree at most $3$, then $\{i+r,j+r\}\subseteq supp(\widetilde{H}(\Sigma ^r(K)))\subseteq \{[\frac{n}{3}], \ldots , n\}$ for some $n$. In particular $i+r\ge [\frac{n}{3}]> \frac{n}{3} -1$, $j+r\le n$, and hence $r>\frac{j-3i-3}{2}$, as we claimed.

Given a class $\mathcal{C}$ of graphs with bounded maximum degree, the homotopy types of the independence complexes of elements of $\mathcal{C}$ is a proper subset of the set of homotopy types of complexes. In other words, we have the following

\begin{coro}
Given a positive integer $n$ there exists a complex $K$ such that for any graph $G$ with $I_G\simeq K$, the maximum degree of $G$ is strictly greater than $n$.
\end{coro}



\bigskip

\textbf{Acknowledgments}
\medskip

I am grateful to Jakob Jonsson for inspiring me with his work to study independence complexes of graphs. The original questions of this article were motivated by a talk he presented in the combinatorics seminar here at KTH. I would also like to thank Anders Bj\"orner and Gabriel Minian for useful discussions and kind advices.

\end{document}